\documentclass[11pt]{article}
\usepackage{geometry,amsmath,amssymb,amsthm,mathtools,booktabs,array,bbm,titlesec,enumitem,xcolor,nicematrix}
\usepackage[colorlinks=true,linkcolor=blue!50!black,citecolor=blue!50!black,urlcolor=blue!50!black]{hyperref}

\newtheorem{theorem}{Theorem}[section]
\newtheorem*{theoremintro}{Theorem}
\newtheorem*{corollaryintro}{Corollary}
\newtheorem{lemma}[theorem]{Lemma}
\newtheorem{proposition}[theorem]{Proposition}
\newtheorem{corollary}[theorem]{Corollary}
\theoremstyle{definition}

\newtheorem{note}[theorem]{Note}

\newcommand{\C}{\mathcal{C}}

\newcommand{\Q}{\mathbb{Q}}

\newcommand{\KK}{\mathbb{K}}
\newcommand{\ff}{\mathfrak f}

\titleformat{\section}[runin]
  {\normalfont\large\bfseries}{\thesection.}{1em}{}
\titlespacing{\section}{0pt}{1.5ex plus .1ex minus .2ex}{1em}

\title{The dimension block of a spherical fusion category}
\author{Andrew Schopieray}
\date{}

\begin{document}

\maketitle

\begin{abstract}
Let $\mathcal{C}$ be a spherical fusion category with global dimension $D$. The Galois conjugates of the dimension character index a block of the $S$-matrix of the Drinfeld center, which we call the dimension block. We prove that for every character $\chi$ of $\mathrm{Gal}(\mathbb{Q}(D)/\mathbb{Q})$, the twisted sum $\sum_\sigma\chi(\sigma)/\sigma(D)$ has absolute value at most $e^{-1}\mathfrak{f}(\chi)^{-1/2}$, where $\mathfrak{f}(\chi)$ is the conductor of $\chi$ and $e$ is the order of the dimensional grading group of $\mathcal{C}$. As an application, we determine the possible global dimensions of spherical fusion categories below $\sqrt{5}$, extending the known classification up to $4\sqrt{3}/5$ of V.\ Ostrik and P.\ Etingof.
\end{abstract}

\vspace{5 mm}

\section{Introduction}\label{sec:intro}

Let $\mathcal{C}$ be a spherical fusion category over the complex numbers in the sense of \cite{MR2183279}. Its global dimension $D:=\dim(\mathcal{C})$ is a totally positive cyclotomic integer \cite{MR2183279,NScong} and an algebraic $d$-number \cite{MR2576705}. It is an open question whether the set
\begin{equation}
  X_s:=\{\mathrm{dim}(\mathcal{C}):\mathcal{C}\text{ a spherical fusion category}\}\subset[1,\infty)
\end{equation}
is discrete \cite{MR2183279,ostrikremarks}. V.\ Ostrik proved that $1$ is an isolated point of $X_s$ \cite[Theorem 4.1.1]{ostrikremarks}, and with P.\ Etingof showed by a computer search that $\tfrac{1}{2}(5-\sqrt{5})$, the global dimension of the Yang--Lee category $\mathcal{YL}$, is the only point of $X_s$ in $(1,\tfrac{4}{5}\sqrt{3})$ \cite[Proposition A.1.1]{ostrikremarks}. He asked whether $\sqrt{2}$ is a limit point of $X_s$, and whether $X_s$ meets the interval $(\tfrac{1}{2}(5-\sqrt{5}),\sqrt2\bigr)$ at all \cite[Questions 1.2.2--1.2.3]{ostrikremarks}. He suggested that the next point of $X_s$ after $\tfrac{1}{2}(5-\sqrt{5})$ is the smallest root $\alpha_3\approx1.84117$ of $t^3-14t^2+49t-49$, the dimension of a Galois conjugate of the rank-three category $\mathcal{C}(\mathfrak{sl}_2,5)_{\mathrm{ad}}$ \cite[Section 4.3]{MR3427429}.

\par Our approach is through the Drinfeld center. The Galois conjugates $\sigma(\dim)$ of the dimension character of $\mathcal{C}$ correspond to simple summands $A_\sigma$ of $I(\mathbbm{1})$, where $I:\mathcal{C}\to\mathcal{Z}(\mathcal{C})$ is induction, and have formal codegrees $\sigma(D)$ \cite{MR3427429}. The block of the normalized $S$-matrix of $\mathcal{Z}(\mathcal{C})$ indexed by these summands, the \emph{dimension block}, is a group matrix for $G_1=\mathrm{Gal}(\mathbb{K}_1/\mathbb{Q})$, where $\mathbb{K}_1$ is the field generated by the dimensions of simple objects of $\mathcal{C}$ (Corollary~\ref{cor:block}). Its eigenvalues are therefore governed by, for $\chi\in\widehat{G_1}$, the character sums
\begin{equation}
\hat{g}_1(\chi):=\sum_{\sigma\in G_1}\chi(\sigma)/\sigma(D).
\end{equation}
Characters of $G_1$ are even primitive Dirichlet characters; write $\mathfrak{f}(\chi)$ for the conductor. The sum $\hat{g}_1(\chi)$ vanishes unless $\chi$ factors through $\mathrm{Gal}(\mathbb{K}_0/\mathbb{Q})$, where $\mathbb{K}_0:=\mathbb{Q}(D)$, and then $\hat{g}_1(\chi)=e\,\hat{g}_{\mathbb{K}_0}(\chi)$, where $\hat{g}_{\KK_0}$ is the same sum over $\mathrm{Gal}(\mathbb{K}_0/\mathbb{Q})$ and $e:=[\mathbb{K}_1:\mathbb{K}_0]$ is the order of the dimensional grading group of $\mathcal{C}$ \cite{MR4836055}.

\begin{theoremintro}[Theorem~\ref{thm:orbit}]
Let $\mathcal{C}$ be a spherical fusion category. Then $|\hat{g}_1(\chi)|^2\leq1/\mathfrak{f}(\chi)$ for every $\chi\in\widehat{G_1}$. Equivalently, $|\hat{g}_{\mathbb{K}_0}(\chi)|^2\leq1/(e^2\mathfrak{f}(\chi))$ for every character $\chi$ of $\mathrm{Gal}(\mathbb{K}_0/\mathbb{Q})$.
\end{theoremintro}

The bound is sharp: equality holds for every nontrivial character when $D$ is any of the numbers $\alpha_3$, $\alpha_5$, $D_m$ below (Lemma~\ref{lem:dimtwo-values}). Already the simplest consequence is useful: Corollary~\ref{cor:YLmin} shows by hand that $\mathrm{dim}(\mathcal{C})\geq\tfrac{1}{2}(5-\sqrt5)$ unless $\mathcal{C}\simeq\mathrm{Vec}$, a fact previously established by a computer search \cite[Proposition A.1.1]{ostrikremarks}.

The proof uses only the congruence property of the modular representation of $\mathcal{Z}(\mathcal{C})$ \cite{NScong} and its compatibility with the Galois action \cite{dong2015congruence}; representations of $\mathrm{SL}_2(\mathbb{Z}/N)$ have been used in a similar spirit to classify modular data \cite{MR4630478}. The vectors $w_\chi=\sum_\sigma\chi(\sigma)\mathbf{e}_{A_\sigma}$ are fixed by the twist and are eigenvectors of the diagonal torus of $\mathrm{SL}_2(\mathbb{Z}/N)$, and $\hat{g}_1(\chi)$ is a matrix coefficient of $\rho(\mathfrak{s})$ between two such vectors. The bound then reduces to a computation in the representations of $\mathrm{SL}_2(\mathbb{Z}/\ell^a)$ induced from the unipotent subgroup (Lemma~\ref{lem:local-pp}).

\par Theorem~\ref{thm:orbit} forces the Galois conjugates of $D$ to spread out as the degree of $\mathbb{K}_0$ grows (Corollary~\ref{cor:conjugates}), so only fields of small degree and small conductor can occur for small $D$. This complements \cite{MR3427429,ostrikremarks,Schopieray2021norm}, where the norms and traces of Galois conjugates of formal codegrees were used to constrain fusion categories. Let $F_n$ and $L_n$ denote the Fibonacci and Lucas numbers, $\phi=\tfrac{1}{2}(1+\sqrt5)$ the golden ratio, and for $m\geq1$ let
\begin{equation}
  D_m:=\sqrt{5}\,(1-\phi^{-2m}),
\end{equation}
the smaller root of $t^2-5F_{2m}t+5(L_{2m}-2)$. The sequence $D_m$ increases to $\sqrt{5}$, and $D_1=\tfrac{1}{2}(5-\sqrt{5})$, $D_2=D_1^2$, $D_3=20-8\sqrt{5}$. Let $\alpha_5\approx2.15585$ be the smallest root of $t^3-42t^2+245t-343$.

\begin{theoremintro}[Theorem~\ref{thm:dimtwo}]
Let $\mathcal{C}$ be a spherical fusion category with $\mathrm{dim}(\mathcal{C})<\sqrt5$. Then
\begin{equation}
      \mathrm{dim}(\mathcal{C})\in\{1,\ 2,\ \alpha_3,\ \alpha_5\}\cup\{D_m:m\geq1\}.
\end{equation}
The seven smallest of these numbers, $1<D_1<\alpha_3<D_2<2<D_3<\alpha_5$, are the global dimensions of spherical fusion categories.
\end{theoremintro}

Whether $D_m\in X_s$ for $m\geq4$ remains open; the first open case is $D_4\approx2.18847$. Table~\ref{tab:dimtwo} lists the known realizations.

\begin{corollaryintro}[Corollary~\ref{cor:ostrik}]
For every $c<\sqrt{5}$ the set $X_s\cap[1,c]$ is finite. In particular $\sqrt{2}$ is not a limit point of $X_s$, the interval $(\tfrac{1}{2}(5-\sqrt{5}),\sqrt{2})$ contains no point of $X_s$, and the smallest point of $X_s$ larger than $\tfrac{1}{2}(5-\sqrt{5})$ is $\alpha_3$.
\end{corollaryintro}

We also note that since Theorem \ref{thm:dimtwo} has bounded global dimensions for spherical fusion categories away from $2$, by passing to the sphericalization of $\mathcal{C}$ \cite[Remark 3.1]{MR2183279} we may conclude that the global dimensions of arbitrary nontrivial fusion categories over the complex numbers are bounded away from $1$.

\begin{corollaryintro}[Corollary~\ref{arbcor}]
    If $\mathcal{C}\not\simeq\mathrm{Vec}$ is a fusion category, then $\dim(\mathcal{C})\geq\tfrac{1}{2}D_3=10-4\sqrt{5}$.
\end{corollaryintro}

\par The organization of the paper is as follows. Section~\ref{sec:prelim} sets up the dimension block. Section~\ref{sec:rep} recalls the modular representation of $\mathcal{Z}(\mathcal{C})$ and locates the vectors $w_\chi$ with the desired eigenvalues. Section~\ref{sec:local} contains the local computation for $\mathrm{SL}_2(\mathbb{Z}/\ell^a)$, and Section~\ref{sec:thm} proves Theorem~\ref{thm:orbit} and its consequences for the conjugates of $D$. Section~\ref{sec:example} works through the proof for a Galois conjugate of the even part of the 2D2 subfactor, whose global dimension is $D_3=20-8\sqrt5$. Section~\ref{app:dimtwo} proves Theorem~\ref{thm:dimtwo} and Corollary~\ref{cor:ostrik}.

\section*{Acknowledgements }We thank V.\ Ostrik for comments on an earlier draft leading to Corollary~\ref{arbcor}. The author acknowledges the use of Claude (Anthropic) in the development of this manuscript and
in particular developing the code for the proof of Lemma \ref{lem:dimtwo-values}. This assistance included suggestions for intermediate claims, proof steps, and corrections. All mathematical statements, proofs, computational results, and references were independently checked by the author, who takes sole responsibility for the contents of the paper.

\section{The dimension block}\label{sec:prelim}

\par Let $\mathcal{C}$ be a spherical fusion category over $\mathbb{C}$ with global dimension $D:=\dim(\mathcal{C})$ and Grothendieck ring $K(\mathcal{C})$, and let $N:=\mathrm{FSexp}(\mathcal{C})$, which is the order of the twist matrix of $\mathcal{Z}(\mathcal{C})$ \cite[Theorem 5.5]{NSexp}. Every fusion category is defined over a number field \cite{MR2183279}; so for $\tau\in\mathrm{Gal}(\overline{\mathbb{Q}}/\mathbb{Q})$, applying $\tau$ to the structure constants of $\mathcal{C}$ gives a spherical fusion category $\mathcal{C}^\tau$, the \emph{Galois conjugate}, with dimension function $\tau(\dim)$ and global dimension $\tau(D)$. Let $F:\mathcal{Z}(\mathcal{C})\to\C$ be the forgetful functor, $I$ its adjoint, and $S$ the unnormalized $S$-matrix of $\mathcal{Z}(\mathcal{C})$. We use the following facts from \cite[Section 2]{MR3427429}. The irreducible representations $E$ of $K(\mathcal{C})\otimes_\mathbb{Z}\mathbb{C}$ have formal codegrees $f_E>0$ with $\sum_E\dim(E)/f_E=1$ \cite[Proposition 2.10]{MR3427429}. There is an injection $E\mapsto A_E$ of $\mathrm{Irr}(K(\mathcal{C}))$ into $\mathcal{O}(\mathcal{Z}(\mathcal{C}))$ whose image is the set of simple summands of $I(\mathbbm{1})$, with $[I(\mathbbm{1}):A_E]=\dim(E)$, $\mathrm{dim} A_E=D/f_E$ and $\theta_{A_E}=1$ \cite[Theorems 2.5 and 2.13]{MR3427429}. Finally, $A_E$ is the simple object whose character $[B]\mapsto S_{A_E,B}/\mathrm{dim} A_E$ of $K(\mathcal{Z}(\mathcal{C}))$ is $\psi_E\circ[F]$ \cite[Example 2.9 and Section 2.3]{MR3427429}; that is,
\begin{equation}\label{eq:ostrik-S}
S_{A_E,Y}=\dim(A_E)\,\psi_E([F(Y)])\quad\text{ for all }\quad Y\in\mathcal{O}(\mathcal{Z}(\mathcal{C})).
\end{equation}
As in the author's work with T.\ Gannon \cite{MR4836055}, consider the data
\begin{equation}
\mathbb{K}_1:=\mathbb{Q}(\mathrm{dim}(X):X\in\mathcal{O}(\mathcal{C})),\qquad G_1:=\mathrm{Gal}(\mathbb{K}_1/\Q),\quad\text{ and }\quad \mathbb{K}_0:=\mathbb{Q}(D)\subseteq \mathbb{K}_1.  
\end{equation}
Put $e:=[\mathbb{K}_1:\mathbb{K}_0]$, the order of the dimensional grading group of $\mathcal{C}$, a power of $2$ \cite{MR4836055}, and $d:=[\mathbb{K}_0:\mathbb{Q}]$. The dimensions of simple objects of $\mathcal{C}$ are real algebraic integers, and so are their conjugates, which are dimensions in the categories $\mathcal{C}^\tau$. They lie in $\mathbb{Q}(\zeta_N)$, since $\dim(X)=\dim(I(X))/D$ and the normalized $S$-matrix of $\mathcal{Z}(\mathcal{C})$, which determines all dimensions in $\mathcal{Z}(\mathcal{C})$ and $D$, has entries in $\mathbb{Q}(\zeta_N)$ \cite[Theorem 7.1]{NScong}. So $\mathbb{K}_1$ is a real abelian field. The characters $\sigma(\dim)$ for $\sigma\in G_1$ are pairwise distinct and we call them the \emph{dimension orbit}. Their formal codegrees are the numbers $\sigma(D)$, each value taken $e$ times. For $\sigma\in G_1$ let $A_\sigma$ be the summand of $I(\mathbbm{1})$ associated to $\sigma(\mathrm{dim})$. These summands are pairwise non-isomorphic, occur with multiplicity one in $I(\mathbbm{1})$, and have dimension $D/\sigma(D)$. We focus on the part of the $S$-matrix indexed by these summands, which we call the \emph{dimension block}. Set
\begin{equation}
  m_1:=\sum_{\sigma\in G_1}\sigma(1/D),\qquad\text{ and }\qquad I_s^{(1)}(\C):=\sum_{\sigma\in G_1}\sigma(1/D)^2.  
\end{equation}
Then $m_1\leq1$, since $m_1$ collects some of the terms of $\sum_E\dim(E)/f_E=1$.

For $m\in(\mathbb{Z}/N)^\times$ let $\sigma_m\in\mathrm{Gal}(\mathbb{Q}(\zeta_N)/\mathbb{Q})$ be given by $\zeta_N\mapsto\zeta_N^m$. Since $\mathbb{K}_1\subseteq\mathbb{Q}(\zeta_N)$, the map $m\mapsto\bar\sigma_m:=\sigma_m|_{\mathbb{K}_1}$ is a surjection $(\mathbb{Z}/N)^\times\to G_1$, so each $\chi\in\widehat{G_1}$ defines a Dirichlet character $\chi(m):=\chi(\bar\sigma_m)$ modulo $N$. We identify $\chi$ with the unique primitive character inducing it; its modulus $\mathfrak{f}(\chi)\mid N$ is the conductor of $\chi$, and it depends only on $\chi$, being the conductor of the fixed field of $\mathrm{ker}\chi$. Since $\sigma_{-1}$ is complex conjugation and $\mathbb{K}_1$ is real, $\chi(-1)=1$, so these characters are even. Define for all $\chi\in\widehat{G_1}$,
\begin{equation}
\hat{g}_1(\chi)=\sum_{\sigma\in G_1}\chi(\sigma)/\sigma(D),\qquad\text{and}\qquad\hat{g}_{\mathbb{K}_0}(\chi)=\sum_{\sigma\in\mathrm{Gal}(\mathbb{K}_0/\mathbb{Q})}\chi(\sigma)/\sigma(D).  
\end{equation}
Then $\hat{g}_1(1)=m_1$ and $\hat{g}_1(\bar\chi)=\overline{\hat{g}_1(\chi)}$. Since $\sigma(D)$ depends only on $\sigma|_{\KK_0}$, we have $\hat{g}_1(\chi)=0$ unless $\chi$ is trivial on $\mathrm{Gal}(\KK_1/\KK_0)$, and in that case $\hat{g}_1(\chi)=e\,\hat{g}_{\KK_0}(\chi)$. Let $s:=\tfrac{S}{D}$ be the normalized $S$-matrix of $\mathcal{Z}(\mathcal{C})$ and $T$ its twist matrix. The following is a consequence of~\eqref{eq:ostrik-S}.

\begin{lemma}\label{lem:rows}
For $\sigma\in G_1$ and $Y\in\mathcal{O}(\mathcal{Z}(\mathcal{C}))$ we have $s_{A_\sigma,Y}=\sigma(\dim(Y)/D)$. In particular $s_{A_\sigma,A_\tau}=(\sigma\tau)(1/D)$, and the row for each $A_\sigma$ is real.
\end{lemma}

\begin{proof}
Since $\mathrm{dim}(Y)=\mathrm{dim}(F(Y))\in \mathbb{K}_1$, \eqref{eq:ostrik-S} with $\psi_E=\sigma(\mathrm{dim})$ gives
\begin{equation}
  S_{A_\sigma,Y}=\frac{D}{\sigma(D)}\sigma(\mathrm{dim}(F(Y)))=\frac{D}{\sigma(D)}\sigma(\mathrm{dim}(Y)).
\end{equation}
Divide by $D$. Since $\mathrm{dim}(A_\tau)/D=\tau(1/D)$, the second claim follows, and the rows are real because $\mathbb{K}_1$ is real.
\end{proof}

Set $\mathcal{Y}:=\{A_\sigma\}_{\sigma\in G_1}$, the set of simple objects indexing the dimension block.

\begin{corollary}\label{cor:block}
The dimension block $s_\mathcal{Y}=(s_{A_\sigma,A_\tau})_{\sigma,\tau}$ satisfies $s_\mathcal{Y} v_\chi=\hat{g}_1(\chi)\,v_{\bar\chi}$, where $v_\chi=(\chi(\tau))_\tau$. Hence its eigenvalues are $m_1$, the numbers $\hat{g}_1(\chi)$ for real $\chi\neq1$, and $\pm|\hat{g}_1(\chi)|$ on $\mathrm{span}(v_\chi,v_{\bar\chi})$ for non-real $\chi$, and
\begin{equation}\label{parsley}
\sum_{\chi\in\widehat{G_1}}|\hat{g}_1(\chi)|^2=|G_1|\,I_s^{(1)}(\mathcal{C}).
\end{equation}
\end{corollary}

\begin{proof}
We compute $(s_\mathcal{Y} v_\chi)(\sigma)=\sum_\tau(\sigma\tau)(1/D)\chi(\tau)=\overline{\chi(\sigma)}\,\hat{g}_1(\chi)$. Parseval's formula applied to the function $\sigma\mapsto\sigma(1/D)$ on $G_1$ gives \eqref{parsley}.
\end{proof}

\section{The representation}\label{sec:rep}

Since $\mathcal{C}$ is spherical, $\mathcal{Z}(\mathcal{C})$ is a modular category whose first Gauss sums both equal $D$ \cite[Theorem 1.2]{MR1966525}, and the modular relations read $(ST)^3=DS^2$ \cite[(1-3) and Section 1.4]{dong2015congruence}. For the normalized modular data this gives $s^2=C$, the duality matrix, and $(sT)^3=s^2$, hence
\begin{equation}\label{eq:sTs}
  sTs=T^{-1}sT^{-1}.
\end{equation}
Let $V=\mathbb C^{\mathcal{O}(\mathcal{Z}(\mathcal{C}))}$ with its standard inner product and recall:
\begin{enumerate}[label=\textup{(R\arabic*)}]
\item The assignments $\rho(\mathfrak{s})=s$ and $\rho(\mathfrak{t})=T$ define a representation of $\mathrm{SL}_2(\mathbb{Z})=\langle\mathfrak{s},\mathfrak{t}\mid\mathfrak{s}^4=1,(\mathfrak{s}\mathfrak{t})^3=\mathfrak{s}^2\rangle$, the canonical modular representation of $\mathcal{Z}(\mathcal{C})$ \cite[(1-10)]{dong2015congruence}. The relations hold by~\eqref{eq:sTs} and $s^2=C$; since both Gauss sums equal $D$, no rescaling of $T$ is needed.
\item $\mathrm{ker}\rho$ is a congruence subgroup of level $N$, so $\rho$ factors through $\mathrm{SL}_2(\mathbb{Z}/N)$, and the entries of $s$ and $T$ lie in $\mathbb{Q}(\zeta_N)$ \cite[Theorems 6.7 and 7.1]{NScong}.
\item For $b\in(\mathbb{Z}/N)^\times$, the element $\mathfrak{h}_b=\mathfrak{t}^b\mathfrak{s}\mathfrak{t}^{b^{-1}}\mathfrak{s}\mathfrak{t}^b\mathfrak{s}^{-1}$ is congruent to $h_b:=\mathrm{diag}(b,b^{-1})$ modulo $N$, and $\rho(\mathfrak{h}_b)=\sigma_b(s)s^{-1}$, where $\sigma_b$ is applied entrywise; this matrix is a signed permutation matrix \cite[Lemma 4.2 and (4-4)]{dong2015congruence}.
\end{enumerate}
Since $\sigma_{-1}$ is complex conjugation on $\mathbb{Q}(\zeta_N)$, taking $b=-1$ in (R3) gives $\bar{s}=\rho(-1)s=Cs$. As $s$ and $C$ are symmetric, $ss^\ast=s\,sC=C^2=1$, so $s$ is unitary and $\rho$ is a unitary representation (see also \cite{BK}).

Write $U:=\langle\mathfrak{t}\rangle$. The diagonal matrices $h_b$, $b\in(\mathbb{Z}/N)^\times$, form a subgroup of $\mathrm{SL}_2(\mathbb{Z}/N)$ isomorphic to $(\mathbb{Z}/N)^\times$, which we call the \emph{torus}, following the terminology for algebraic groups. Since $h_b\mathfrak{t} h_b^{-1}=\mathfrak{t}^{b^2}$, the torus normalizes $U$, so it preserves the space $V^U$ of twist-fixed vectors, and $V^U$ splits into eigenspaces $V^U_\alpha$ on which $\rho(h_b)$ acts by $\alpha(b)$, one for each character $\alpha$ of $(\mathbb{Z}/N)^\times$. For $\chi\in\widehat{G_1}$ define
\begin{equation}
  w_\chi:=\sum_{\sigma\in G_1}\chi(\sigma)\,\mathbf e_{A_\sigma}\in V.
\end{equation}

\begin{lemma}\label{lem:lagvec}
The vector $w_\chi$ has the following properties.
\begin{enumerate}[label=\textup{(\alph*)}]
\item $Tw_\chi=w_\chi$.
\item $\rho(h_b)w_\chi=\chi(b)\,w_\chi$ for every $b\in(\mathbb{Z}/N)^\times$.
\item $\langle w_{\bar\chi},s\,w_\chi\rangle=|G_1|\,\hat{g}_1(\chi)$.
\end{enumerate}
\end{lemma}

\begin{proof}
For (a), every $A_\sigma$ has trivial twist. For (b), put $Q:=\rho(h_b)=\sigma_b(s)s^{-1}$, using (R3). By Lemma~\ref{lem:rows} the $A_\tau$-row of $\sigma_b(s)$ is the $A_{\bar\sigma_b\tau}$-row of $s$, so $(Qv)_{A_\tau}=v_{A_{\bar\sigma_b\tau}}$ for every $v$, and the $\mathcal{Y}$-coordinates of $Qw_\chi$ are those of $\chi(b)w_\chi$. Since $Q$ is unitary, $Qw_\chi$ has the same norm as $w_\chi$, so its other coordinates vanish and $Qw_\chi=\chi(b)w_\chi$. For (c) we compute $(sw_\chi)_{A_\kappa}=\sum_\tau(\kappa\tau)(1/D)\chi(\tau)=\overline{\chi(\kappa)}\,\hat{g}_1(\chi)$, and pair with $w_{\bar\chi}$.
\end{proof}

\section{The local computation}\label{sec:local}

Let $M=\ell^a$ be a prime power and $\Gamma_M=\mathrm{SL}_2(\mathbb{Z}/M)$. Let $U\subseteq\Gamma_M$ be generated by $\mathfrak{t}$, and let $\{h_b\}\cong(\mathbb{Z}/M)^\times$ be the torus, as in Section~\ref{sec:rep}. For a unitary representation $(\pi,W)$ of $\Gamma_M$ let $P$ be the orthogonal projection onto $W^{U}$, and for a character $\alpha$ of $(\mathbb{Z}/M)^\times$ put
\begin{equation}
  W^U_\alpha=\{w\in W^U:\pi(h_b)w=\alpha(b)w\ \text{for all}\ b\}.
\end{equation}
We call $P\pi(\mathfrak{s})P$, viewed as an operator on $W^U$, the \emph{compression} of $\pi(\mathfrak{s})$ to $W^U$; it is the block of $\pi(\mathfrak{s})$ indexed by $W^U$, and has norm at most $1$. For $w,w'\in W^U$ we have $\langle w',\pi(\mathfrak{s})w\rangle=\langle w',P\pi(\mathfrak{s})Pw\rangle$, so matrix coefficients between $U$-fixed vectors only see the compression. Because $h_b$ normalizes $U$ and $\mathfrak{s} h_b\mathfrak{s}^{-1}=h_{b^{-1}}$, the compression maps $W^U_\alpha$ into $W^U_{\bar\alpha}$. Let $\kappa_M(\alpha)$ be the supremum of $\|P\pi(\mathfrak{s})P|_{W^U_\alpha}\|$ over all unitary representations of $\Gamma_M$.

\begin{lemma}\label{lem:local-pp}
Let $\alpha$ be a nontrivial character of $(\mathbb{Z}/M)^\times$ of conductor $\ell^c$, where $1\le c\le a$. Then $\kappa_M(\alpha)=\ell^{-c/2}=\mathrm{cond}(\alpha)^{-1/2}$. For $\alpha=1$ we have $\kappa_M(1)=1$, attained by the trivial representation.
\end{lemma}

\begin{proof}
Put $R:=\mathbb{Z}/M$ and $\Gamma:=\Gamma_M$. Let $\Omega\subset R^2$ be the set of primitive column vectors, those with a unit coordinate. The group $\Gamma$ acts transitively on $\Omega$, and $U$ is the stabilizer of $(1,0)^\mathrm{T}$. So $\mathbb{C}[\Omega]$, with the action $(g\cdot f)(v)=f(g^{-1}v)$ and the counting inner product, is $\mathrm{Ind}_U^\Gamma1$. A unitary representation is an orthogonal sum of irreducibles, and $P$, $\mathfrak{s}$ and $h_b$ preserve each summand. By Frobenius reciprocity, every irreducible representation with nonzero $U$-fixed vectors embeds isometrically, up to scaling, into $\mathbb C[\Omega]$. So it suffices to bound the compression on $\mathbb C[\Omega]^U_\alpha$. We now define some local terminology to make the proof clearer.

\noindent \emph{Cells.} Write $v:=(p,q)$. Since $u_x(p,q)=(p+xq,q)$, the $U$-fixed functions are the functions constant on $U$-orbits. There are two kinds of orbits:
\begin{itemize}
\item\emph{The big cell.} For $q\in R^\times$ the orbit is $O_q:=\{(p,q):p\in R\}$.
\item\emph{The small cells.} For $q\notin R^\times$, so that $p\in R^\times$, the orbit is $\{(p',q):p'\in p+qR\}$.
\end{itemize}
The torus acts by $(h_bf)(p,q)=f(b^{-1}p,bq)$ and preserves both kinds of cells. On $\mathrm{span}\{1_{O_q}\}$ it acts as the regular representation of $R^\times$. So the $\alpha$-eigenspace of the big cell is spanned by
\begin{equation}
  F_\alpha(p,q)=\alpha(q)\,[q\in R^\times],\qquad \|F_\alpha\|^2=M\,|R^\times|.
\end{equation}
Thus $\mathbb C[\Omega]^U_\alpha=\mathbb CF_\alpha\oplus\Sigma_\alpha$, where $\Sigma_\alpha$ consists of the functions supported on the small cells.

\noindent \emph{Big cell to small cells.} Since $\mathfrak{s}^{-1}(p,q)=(q,-p)$, we have $(\mathfrak{s} F_\alpha)(p,q)=\alpha(-p)[p\in R^\times]$. Now average over $U$-orbits:
\begin{itemize}
\item On $O_q$ the average is $M^{-1}\sum_{p\in R^\times}\alpha(-p)=0$.
\item On the small-cell orbit of $(p,q)$ with $qR=\ell^kR$, $k\ge1$, the average is $\alpha(-p)$ times the mean of $\alpha$ over $1+\ell^kR$. That mean is $1$ if $k\ge c$ and $0$ otherwise.
\end{itemize}
Hence
\begin{equation}
  (P\mathfrak{s}PF_\alpha)(p,q)=\alpha(-p)\,[q\in\ell^cR],\qquad \|P\mathfrak{s} PF_\alpha\|^2=|R^\times|\,|\ell^cR|=\ell^{-c}\|F_\alpha\|^2,
\end{equation}
and $P\mathfrak{s} PF_\alpha\in\Sigma_{\bar\alpha}$.

\noindent \emph{Small cells to big cell.} If $f$ is supported on $\{q\notin R^\times\}$, then $\mathfrak{s}f$ is supported on $\{q\in R^\times\}$. So $P\mathfrak{s}P$ maps $\Sigma_\alpha$ into the $\bar\alpha$-part of the big cell, which is $\mathbb{C}F_{\bar\alpha}$. Call this map $B$. Since $\mathfrak{s}^{-1}=h_{-1}\mathfrak{s}$ with $h_{-1}$ central, $(P\mathfrak{s} P)^\ast=h_{-1}P\mathfrak{s}P$. So $B^\ast$ is the restriction of $h_{-1}P\mathfrak{s} P$ to $\mathbb{C}F_{\bar\alpha}$, which has norm $\ell^{-c/2}$ by the previous step applied to $\bar\alpha$. Hence $\|B\|=\ell^{-c/2}$.

\noindent \emph{Conclusion.} Let $w=xF_\alpha+y$ with $y\in\Sigma_\alpha$. The images $xP\mathfrak{s} PF_\alpha\in\Sigma_{\bar\alpha}$ and $P\mathfrak{s} Py\in\mathbb{C}F_{\bar\alpha}$ are orthogonal, so
\begin{equation}
  \|P\mathfrak{s}Pw\|^2\le\ell^{-c}\bigl(|x|^2\|F_\alpha\|^2+\|y\|^2\bigr)=\ell^{-c}\|w\|^2,
\end{equation}
with equality for $w=F_\alpha$. The case $\alpha=1$ is clear, since $\|P\pi(\mathfrak{s})P\|\leq1$.
\end{proof}

\section{The theorem and its corollaries}\label{sec:thm}

\begin{theorem}\label{thm:orbit}
Let $\C$ be a spherical fusion category. Then for every $\chi\in\widehat{G_1}$,
\begin{equation}
|\hat{g}_1(\chi)|^2\leq\frac1{\mathfrak{f}(\chi)} .
\end{equation}
Equivalently, $|\hat{g}_{\mathbb{K}_0}(\chi)|^2 \leq 1/(e^2\mathfrak{f}(\chi))$ for every character $\chi$ of $\mathrm{Gal}(\mathbb{K}_0/\mathbb{Q})$.
\end{theorem}

\begin{proof}
By (R2) and the Chinese remainder theorem, $\rho$ is a unitary representation of $\prod_{\ell\mid N}\mathrm{SL}_2(\mathbb{Z}/\ell^{a_\ell})$, where $N=\prod_\ell\ell^{a_\ell}$, and $\mathfrak{s}$, $\mathfrak{t}$, $h_b$ correspond to the tuples of their images. Every irreducible constituent has the form $\pi=\bigotimes_\ell\pi_\ell$. A product of $\ell^{a_\ell}$-th roots of unity for distinct primes $\ell$ is $1$ only if every factor is $1$, so
\begin{equation}
  \pi^U=\bigotimes_\ell\pi_\ell^{U_\ell}.
\end{equation}
The $\chi$-eigenspace of the torus is the tensor product of the $\chi_\ell$-eigenspaces, where $\chi=\prod\chi_\ell$ and $\mathrm{cond}(\chi_\ell)=\ell^{v_\ell(\ff)}$; here $\ff=\ff(\chi)$ divides $N$ because $\mathbb{K}_1\subseteq\mathbb{Q}(\zeta_N)$. The compression of $\rho(\mathfrak{s})$ is the tensor product of the local compressions, which by Lemma~\ref{lem:local-pp} have norm at most $\ell^{-v_\ell(\mathfrak{f})/2}$ at primes $\ell\mid\mathfrak{f}$ and at most $1$ elsewhere. Hence $P\rho(\mathfrak{s})P$ has norm at most $\mathfrak{f}^{-1/2}$ on $V^U_\chi$. By Lemma~\ref{lem:lagvec}, the vectors $w_\chi\in V^U_\chi$ and $w_{\bar\chi}\in V^U_{\bar\chi}$ both have norm $|G_1|^{1/2}$. Therefore
\begin{equation}
  |G_1|\,|\hat{g}_1(\chi)|=|\langle w_{\bar\chi},P\rho(\mathfrak{s})Pw_\chi\rangle|\le\mathfrak{f}^{-1/2}\,|G_1|.
\end{equation}
The second statement follows from $\hat{g}_1=e\,\hat{g}_{\mathbb{K}_0}$.
\end{proof}

For a real abelian field $\mathbb{K}$ put $J(\mathbb{K}):=\sum_\chi1/\mathfrak{f}(\chi)$, the sum over the characters of $\mathrm{Gal}(\mathbb{K}/\mathbb{Q})$. Since the conductor of every nontrivial even character is at least $5$, we have $J(\KK)-1\le([\KK:\Q]-1)/5$.

\begin{corollary}\label{cor:orbit-is}
The Galois conjugates of $D$ satisfy
\begin{align}
\sum_{\sigma\in\mathrm{Gal}(\KK_0/\Q)}\sigma(D)^{-2}&\le\frac{J(\KK_0)}{e^2d},&&\text{equivalently}&I_s^{(1)}(\C)&\leq\frac{J(\KK_0)}{ed},\label{eq:Is}\\
\sum_{\tau\in\mathrm{Gal}(\mathbb{K}_0/\Q)}\bigl(\tau(1/D)-\bar x\bigr)^2&\le\frac{J(\KK_0)-1}{e^2d},&&\text{where}&\bar x&=\frac{m_1}{ed}.\label{eq:variance}
\end{align}
\end{corollary}

\begin{proof}
By Parseval's identity over $\mathrm{Gal}(\mathbb{K}_0/\mathbb{Q})$,
\begin{equation}
d\sum_\sigma\sigma(1/D)^2=\sum_\chi|\hat{g}_{\mathbb{K}_0}(\chi)|^2,\qquad d\sum_\tau(\tau(1/D)-\bar x)^2=\sum_{\chi\neq1}|\hat{g}_{\mathbb{K}_0}(\chi)|^2,
\end{equation}
since $\bar{x}$ is the mean of the numbers $\tau(1/D)$ and $\hat{g}_{\mathbb{K}_0}(1)=m_1/e$. Apply Theorem~\ref{thm:orbit}. Each value $\sigma(1/D)$ occurs $e$ times in $I_s^{(1)}(\C)$.
\end{proof}

\begin{corollary}\label{cor:conjugates}
Every Galois conjugate of $D=\dim(\C)$ satisfies
\begin{equation}
  \sigma(D)\geq\frac{ed}{m_1+\sqrt{(d-1)(J(\KK_0)-1)}}\geq\frac{ed\sqrt{5}}{\sqrt5+d-1}.  
\end{equation}
\end{corollary}

\begin{proof}
For $\tau\in\mathrm{Gal}(\mathbb{K}_0/\mathbb{Q})$ put $x_\tau=\tau(1/D)$. Since the numbers $x_\tau-\bar x$ sum to $0$, the Cauchy--Schwarz inequality applied to the other $d-1$ of them gives $(x_\tau-\bar x)^2\le(d-1)\bigl(\sum_\tau(x_\tau-\bar x)^2-(x_\tau-\bar x)^2\bigr)$, that is, $(x_\tau-\bar x)^2\le\frac{d-1}d\sum_\tau(x_\tau-\bar x)^2$. By~\eqref{eq:variance}, every $x_\tau$ is at most $\bar x+\sqrt{(d-1)(J(\KK_0)-1)}/(de)$. This is the first inequality. The second follows from $m_1\le1$ and $J(\KK_0)-1\le(d-1)/5$.
\end{proof}

\begin{corollary}\label{cor:YLmin}
If $\mathcal{C}\not\simeq\mathrm{Vec}$, then $\dim(\mathcal{C})\geq\tfrac{1}{2}(5-\sqrt{5})$, with equality only if $\mathbb{K}_0=\mathbb{Q}(\sqrt{5})$, $e=1$ and $m_1=1$.
\end{corollary}

\begin{proof}
If $D\in\mathbb{Q}$, then $D$ is a rational algebraic integer, and $D\neq1$ by \cite[Theorem 4.1.1]{ostrikremarks}, so $D\geq2$. Otherwise $d\geq2$. The second bound of Corollary~\ref{cor:conjugates} increases with $e$ and $d$, and at $e=1$, $d=2$ it equals $2\sqrt5/(\sqrt5+1)=\tfrac{1}{2}(5-\sqrt5)$. Equality requires $e=1$, $d=2$ and equality in both inequalities of Corollary~\ref{cor:conjugates}, that is, $m_1=1$ and $J(\mathbb{K}_0)-1=\frac{1}{5}$, which forces $\mathbb{K}_0=\mathbb{Q}(\sqrt{5})$.
\end{proof}

\begin{note}
Corollary~\ref{cor:YLmin} shows ``by hand'' that $\tfrac{1}{2}(5-\sqrt{5})$, the global dimension of the Yang--Lee category, is the smallest global dimension of a spherical fusion category other than $1$; this also follows from the computer search in \cite[Proposition A.1.1]{ostrikremarks}.
\end{note}

\begin{note}
Sphericality enters the proof of Theorem~\ref{thm:orbit} through the modularity of $\mathcal{Z}(\mathcal{C})$ with first Gauss sum $D$ \cite{MR1966525}, through the reality of the dimensions (Lemma~\ref{lem:rows}), and through the description of $I(\mathbbm{1})$ in \cite{MR3427429}.
\end{note}

\section{An example: the even part of 2D2}\label{sec:example}

We follow the proof of Theorem~\ref{thm:orbit} through a category of global dimension $D_3$. Let $\mathcal{C}$ be the Galois conjugate of the even part of the 2D2 subfactor \cite{MR4939540,Izu18,MP15} in which the dimension of the non-invertible simple objects is negative. It has simple objects $\mathbbm{1},\alpha,\eta,\alpha\eta$ (denoted $\mathbbm{1},\alpha,\rho,\alpha\rho$ in \cite{MR4939540}) with
\begin{equation}
\alpha\otimes\alpha\cong\mathbbm{1},\qquad\eta\otimes\eta\cong\mathbbm{1}\oplus2\eta\oplus2\alpha\eta,
\end{equation}
and we write $\lambda:=\dim(\eta)=2-\sqrt{5}$ and $D=2+2\lambda^2=20-8\sqrt5=D_3\approx2.11146$. Here $\KK_1=\KK_0=\Q(\sqrt5)$, so $e=1$, $d=2$ and $G_1=\{1,\sigma\}$ with $\sigma(\sqrt5)=-\sqrt5$. The ring $K(\C)$ has four characters: the dimension orbit $\dim$ and $\sigma(\dim)=\mathrm{FPdim}$, with formal codegrees $D$ and $\sigma(D)=20+8\sqrt5$, and two rational characters $\psi_\pm$ with $\psi_\pm(\alpha)=-1$ and $\psi_\pm(\eta)=\pm1$, each with formal codegree $4$. The characters of $G_1$ are the trivial character and $\chi=\bigl(\frac{\cdot}{5}\bigr)$, of conductor $5$. The data of $\mathcal{Z}(\mathcal{C})$ used below is the Galois conjugate of that computed in \cite[Section 3.2]{MR4939540}. Directly from the definitions,
\begin{equation}
\hat{g}_1(1)=m_1=\frac1D+\frac1{\sigma(D)}=\frac12\quad\text{ and }\quad \hat{g}_1(\chi)=\frac1D-\frac1{\sigma(D)}=\frac{16\sqrt5}{80}=\frac1{\sqrt5}.
\end{equation}
So $|\hat{g}_1(\chi)|^2=1/\ff(\chi)$, and Theorem~\ref{thm:orbit} holds with equality for $\chi$. For the trivial character the bound is strict: $m_1=\frac12<1$, and the remaining $\frac12=\frac14+\frac14$ of $\sum_E\dim(E)/f_E=1$ comes from the two formal codegrees equal to $4$. By \cite[Section 3.2]{MR4939540},
\begin{equation}
  I(\mathbbm{1})\cong X_0\oplus X_1\oplus X_2\oplus X_3,
\end{equation}
with $\dim(X_0)=1$, $\dim(X_1)=9-4\sqrt{5}=\tfrac{D}{\sigma(D)}$, and $\dim(X_2)=\dim(X_3)=5-2\sqrt{5}=\tfrac{D}{4}$, all with trivial twist. So $A_1=X_0$, $A_\sigma=X_1$, and $X_2$, $X_3$ are the summands $A_{\psi_\pm}$. By~\eqref{eq:ostrik-S} and Lemma~\ref{lem:rows}, the block of $s$ indexed by $X_0,X_1,X_2,X_3$ is
\begin{equation}\label{eq:2D2-s}
  \begin{pNiceMatrix}[margin,columns-width=auto,cell-space-limits=2pt]
  \Block[draw,rounded-corners]{2-2}{}
  1/D&1/\sigma(D)&1/4&1/4\\
  1/\sigma(D)&1/D&1/4&1/4\\
  1/4&1/4&\ast&\ast\\
  1/4&1/4&\ast&\ast
  \end{pNiceMatrix},
\end{equation}
with the dimension block boxed. Only the first two rows are needed for the proof. The entries $\ast$ are the numbers $\frac14\psi_\pm(F(X_j))$ for $j=2,3$, which depend on the restrictions $F(X_2)$, $F(X_3)$ described in \cite[Lemma 3.14]{MR4939540}. The boxed block has eigenvectors $(1,1)$ and $(1,-1)$ with eigenvalues $m_1=\frac12$ and $\hat{g}_1(\chi)=1/\sqrt5$, as in Corollary~\ref{cor:block}.

The summands $Y_i$ of $I(\alpha)$ satisfy $\theta_{Y_i}^2=\nu_2(\alpha)$, where $\nu_2(\alpha)=-1$ is the second Frobenius--Schur indicator of $\alpha$, denoted $\lambda_\alpha$ in \cite[Section 3.2 and Theorem 3.43]{MR4939540}. With $\sqrt5\in\Q(\zeta_N)$ this gives $20\mid N$, so in the proof of Theorem~\ref{thm:orbit} the representation $\rho$ has at least two local factors. Since $\chi$ has conductor $5$, the bound comes from the factor at $\ell=5$, where the compression has norm at most $5^{-1/2}$ by Lemma~\ref{lem:local-pp}; the other factors contribute at most $1$. Since $\psi_\pm$ are rational-valued, the rows of $s$ at $X_2$ and $X_3$ are rational by~\eqref{eq:ostrik-S}, and the proof of Lemma~\ref{lem:lagvec}(b) shows that $\rho(h_b)$ fixes $\mathbf e_{X_2}$ and $\mathbf e_{X_3}$. The same proof shows that $\rho(h_b)$ fixes or exchanges $\mathbf e_{X_0}$ and $\mathbf e_{X_1}$ according as $\chi(b)=1$ or $-1$. Hence
\begin{equation}
  w_1=\mathbf e_{X_0}+\mathbf e_{X_1},\ \mathbf e_{X_2},\ \mathbf e_{X_3}\in V^U_1,\qquad w_\chi=\mathbf e_{X_0}-\mathbf e_{X_1}\in V^U_\chi .
\end{equation}
Other twist-fixed simple objects of $\mathcal{Z}(\mathcal{C})$ may also contribute to $V^U$; we do not need them.

By Lemma~\ref{lem:rows}, the coordinate of $s\,w_\chi$ at a simple object $Y$ is $\dim(Y)/D-\sigma\bigl(\dim(Y)/D\bigr)$. On the summands of $I(\mathbbm{1})$ this gives
\begin{equation}
(s\,w_\chi)_{X_0}=\frac1{\sqrt5},\quad (s\,w_\chi)_{X_1}=-\frac1{\sqrt5},\quad (s\,w_\chi)_{X_2}=(s\,w_\chi)_{X_3}=0,
\end{equation}
so that $\langle w_\chi,s\,w_\chi\rangle=2/\sqrt{5}=|G_1|\,\hat{g}_1(\chi)$; the coordinates at $X_2$ and $X_3$ vanish because $\dim(X_2)/D=\frac{1}{4}$ is rational. The proof of Theorem~\ref{thm:orbit} bounds more than the matrix coefficient: it gives $\|P\rho(\mathfrak{s})P\,w_\chi\|^2\leq\mathfrak{f}(\chi)^{-1}\|w_\chi\|^2=\frac{2}{5}$. The two coordinates at $X_0$ and $X_1$ already account for $\frac{1}{5}+\frac{1}{5}=\frac{2}{5}$. Therefore $s\,w_\chi$ has no component at any other simple object with trivial twist, that is,
\begin{equation}
  \theta_Y=1,\ Y\notin\{X_0,X_1\}\quad\Longrightarrow\quad \frac{\mathrm{dim}(Y)}D\in\Q .
\end{equation}
This holds for $X_2$ and $X_3$, and the proof predicts it for every other twist-fixed simple object of $\mathcal{Z}(\mathcal{C})$ without computing them. The remaining $\frac85$ of $\|s\,w_\chi\|^2=2$ lies on simple objects with nontrivial twist. By contrast, the coordinates of $s\,w_1$ at $X_0,X_1,X_2,X_3$ are all $\frac12$, so $w_1$ is not fixed by $s$; this reflects $m_1<1$. With $J(\KK_0)=1+\frac15$ and the actual value $m_1=\frac12$, the first inequality of Corollary~\ref{cor:conjugates} is an equality:
\begin{equation}
  D\ge\frac{2}{\frac12+\frac1{\sqrt5}}=\frac{4\sqrt5}{\sqrt5+2}=20-8\sqrt5 .
\end{equation}
The second inequality, which replaces $m_1$ by $1$, only gives $D\geq\tfrac{1}{2}(5-\sqrt{5})$. In the notation of Lemma~\ref{lem:dimtwo-sqrt5}, $D=\tfrac{1}{2}(u-k\sqrt{5})$ with $u=40$, $k=16$ and $n=DD'=80=5k$, and $(k+2)^2-5\cdot8^2=4$ is the solution $L_6=18$, $F_6=8$ of the Pell equation, so $D=D_3$.

\section{Spherical fusion categories of small global dimension}\label{app:dimtwo}

We keep the notation of Section~\ref{sec:prelim}, and recall $\alpha_3$, $\alpha_5$ and $D_m$ from Section~\ref{sec:intro}. The Galois conjugate of $D_m$ is $\sqrt{5}\,(\phi^{2m}-1)$, so $D_m$ is a $d$-number with $m_1=F_{2m}/(L_{2m}-2)\le1$.

\begin{theorem}\label{thm:dimtwo}
Let $\mathcal{C}$ be a spherical fusion category with $\mathrm{dim}(\mathcal{C})<\sqrt5$. Then
\begin{equation}
  \mathrm{dim}(\mathcal{C})\in\{1,\ 2,\ \alpha_3,\ \alpha_5\}\cup\{D_m:m\ge1\}.
\end{equation}
Moreover, the seven smallest values $1$, $D_1$, $\alpha_3$, $D_2$, $2$, $D_3$, $\alpha_5$ are the global dimensions of spherical fusion categories (Proposition~\ref{prop:realize}).
\end{theorem}

The proof occupies the rest of this section. If $D\in\mathbb{Q}$, then $D$ is a rational algebraic integer, so $D\in\{1,2\}$. Now assume $D$ is irrational, and first assume that $D$ is the smallest of its Galois conjugates; this assumption is removed at the end.

\begin{lemma}\label{lem:dimtwo-fields}
Let $D<\sqrt{5}$ be irrational and the smallest of its conjugates. Then $e=1$, and $\mathbb{K}_0$ is one of the following:
\begin{enumerate}
\item a real quadratic field;
\item a cyclic cubic field of conductor $7$, $9$, $13$, $19$ or $31$;
\item a real abelian quartic field, in which case $J(\mathbb{K}_0)-1\leq\frac{7}{20}$.
\end{enumerate}
\end{lemma}

\begin{proof}
The second bound in Corollary~\ref{cor:conjugates} is increasing in $e$ and $d$. For $e\geq2$ and $d\geq2$ it gives $D\geq4\sqrt{5}/(\sqrt{5}+1)\approx2.764$. So $e=1$. For $d\geq3$, the first inequality of Corollary~\ref{cor:conjugates} together with $m_1\leq1$ and $D<\sqrt{5}$ gives
\begin{equation}\label{eq:dimtwo-J}
  (d-1)(J(\mathbb{K}_0)-1)\geq(\frac{d}{\sqrt5}-1)^2 .
\end{equation}
The conductors of the even primitive Dirichlet characters begin, in increasing order and with multiplicity, $5,7,7,8,9,9,11,11,11,11,12,13,\dots$. Since $\mathbb{K}_0$ has $d-1$ distinct nontrivial characters, $J(\mathbb{K}_0)-1$ is at most the sum of the reciprocals of the first $d-1$ entries of this list. We call this the \emph{crude bound}.
\begin{itemize}
\item For $8\leq d\leq30$ the crude bound shows that~\eqref{eq:dimtwo-J} fails. For $d\geq31$, note that a conductor $\mathfrak{f}\geq3$ carries at most $(\mathfrak{f}-1)/2$ even primitive characters, so fewer than $x^2/4$ have conductor at most $x$. Hence the $j$th smallest conductor among the nontrivial characters of $\mathbb{K}_0$ exceeds $2\sqrt {j}$, and $J(\mathbb{K}_0)-1<\sum_{j<d}(2\sqrt {j})^{-1}\leq\sqrt{d-1}$. The left side of~\eqref{eq:dimtwo-J} is then less than $(d-1)^{3/2}$, which at $d=31$ is $164.3$ against $165.5$ on the right, and the ratio of the two sides decreases in $d$ because its logarithmic derivative $\frac{3}{2(d-1)}-\frac{2}{d-\sqrt{5}}$ is negative.
\item For $d=5$ and $d=7$ the field is cyclic of prime degree. Its nontrivial characters share one conductor, which is at least $11$ and $29$ respectively. So $J(\mathbb{K}_0)-1$ is at most $\frac{4}{11}<0.364$ and $\frac{6}{29}<0.207$. The right side of~\eqref{eq:dimtwo-J} divided by $d-1$ is $0.382$ and $0.757$.
\item For $d=6$ the field is cyclic. It has one quadratic character $\chi^3$ of conductor $\mathfrak{f}_2$, two cubic characters $\chi^{\pm2}$ of conductor $\mathfrak{f}_3$, and two characters $\chi^{\pm1}$ of order $6$ of conductor $\mathfrak{f}_6$. Since $\mathfrak{f}_2$ and $\mathfrak{f}_3$ divide $\mathfrak{f}_6$, the conductor $\mathfrak{f}_6$ is a multiple of $\mathrm{lcm}(\mathfrak{f}_2,\mathfrak{f}_3)$. If $\mathfrak{f}_2=5$, then $5\nmid\mathfrak{f}_3$, because $(\mathbb{Z}/5^k)^\times$ has no element of order $3$; so $\mathfrak{f}_6\geq5\mathfrak{f}_3$ and $J(\mathbb{K}_0)-1\leq\frac{1}{5}+\frac{12}{5\mathfrak{f}_3}\le\frac{1}{5}+\frac{12}{35}<0.543$. If $\mathfrak{f}_2\ge8$, then $\mathrm{lcm}(\mathfrak{f}_2,\mathfrak{f}_3)\geq13$, because $\mathfrak{f}_3\in\{7,9,13,\dots\}$ and $\mathfrak{f}_2\in\{8,12,13,\dots\}$, so $J(\mathbb{K}_0)-1\leq\frac{1}{8}+\frac{2}{7}+\frac{2}{13}<0.565$. The right side of~\eqref{eq:dimtwo-J} divided by $5$ is $0.5666$.
\item For $d=3$ both nontrivial characters have the conductor $\mathfrak{f}$ of $\mathbb{K}_0$, and~\eqref{eq:dimtwo-J} reads $4/\mathfrak{f}\geq(3/\sqrt{5}-1)^2\approx0.1167$. So $\mathfrak{f}\leq34$. The cyclic cubic conductors up to $34$ are $7$, $9$, $13$, $19$ and $31$.
\end{itemize}

It remains to bound $J(\mathbb{K}_0)$ for real abelian quartic fields.
\begin{itemize}
\item If $\mathbb{K}_0$ is cyclic, its quadratic character has conductor at least $5$, and its two characters of order $4$ are even, so have conductor at least $15$: the quartic characters of conductors $5$ and $13$ are odd. So $J(\mathbb{K}_0)-1\leq\frac{1}{5}+\frac{2}{15}=\frac{1}{3}$.
\item If $\mathbb{K}_0$ is biquadratic and does not contain $\mathbb{Q}(\sqrt{5})$, its three quadratic subfields have distinct discriminants different from $5$, so $J(\mathbb{K}_0)-1\leq\frac{1}{8}+\frac{1}{12}+\frac{1}{13}<\frac{1}{3}$.
\item If $\mathbb{K}_0$ is biquadratic and contains $\Q(\sqrt{5})$, its other two quadratic subfields have discriminants $\delta$ and $5\delta$, for a fundamental discriminant $\delta\geq8$ prime to $5$. So $J(\mathbb{K}_0)-1=\frac{1}{5}+\frac{6}{5\delta}\leq\frac{7}{20}$.
\end{itemize}
In the quartic case~\eqref{eq:dimtwo-J} only requires $J(\mathbb{K}_0)-1\geq(4/\sqrt{5}-1)^2/3\approx0.2075$, which does not exclude any of the three types.
\end{proof}

\begin{lemma}\label{lem:dimtwo-sqrt5}
Let $D<\sqrt{5}$ be a $d$-number with $\mathbb{Q}(D)=\mathbb{Q}(\sqrt{5})$ which is smaller than its conjugate $D'$, and suppose that $\sqrt{5}\,(1/D-1/D')\leq1$. Then $D=D_m$ for some $m\geq1$.
\end{lemma}

\begin{proof}
Write $D=\tfrac{1}{2}(u-k\sqrt{5})$ with integers $u,k\geq1$, and put $n=DD'=(u^2-5k^2)/4$, so that the minimal polynomial of $D$ is $t^2-ut+n$. Since $1/D-1/D'=k\sqrt{5}/n$, the hypothesis reads $5k/n\leq1$, that is $n\geq5k$. And $D<\sqrt{5}$ gives $u<(k+2)\sqrt{5}$, hence $u^2<5k^2+20k+20$, that is $n<5k+5$. Now $u^2=4n+5k^2$, so the $d$-number condition $n\mid u^2$ \cite{MR2576705} says $n\mid5k^2$. Therefore $q=5k^2/n$ is an integer with
\begin{equation}
  k-1<\frac{5k^2}{5k+4}\leq q\leq\frac{5k^2}{5k}=k ,
\end{equation}
so $q=k$ and $n=5k$. Then $u^2=5k(k+4)$, so $5\mid u$, and $u=5w$ turns this into the Pell equation $(k+2)^2-5w^2=4$. Its positive solutions are $k+2=L_{2m}$ and $w=F_{2m}$ with $m\geq1$, so $u=5F_{2m}$ and $n=5(L_{2m}-2)$. This is the minimal polynomial of $D_m$.
\end{proof}

For an irrational $D$ with $e=1$ we call $v(D):=\max_{\chi\neq1}\mathfrak{f}(\chi)\,|\hat{g}_{\mathbb{K}_0}(\chi)|^2$ its \emph{value}; Theorem~\ref{thm:orbit} says $v(D)\leq1$.

\begin{lemma}\label{lem:dimtwo-values}
Let $D<\sqrt{5}$ be irrational and the smallest of its conjugates. Then $D$ is $\alpha_3$, $\alpha_5$, or $D_m$ for some $m\geq1$, and it attains equality in Theorem~\ref{thm:orbit} for every nontrivial character.
\end{lemma}

\begin{proof}
By Lemma~\ref{lem:dimtwo-fields}, $e=1$ and $\mathbb{K}_0$ is quadratic, cubic or quartic, and $D$ is a $d$-number \cite{MR2576705}. Put $\beta_\tau=\tau(D)^{-1}$ for $\tau\in\mathrm{Gal}(\mathbb{K}_0/\mathbb{Q})$, and let $\bar\beta$ be their mean. By~\eqref{eq:variance},
\begin{equation}\label{eq:spread}
  \frac{1}{2}(\max_\tau\beta_\tau-\min_\tau\beta_\tau)^2\le\sum_\tau(\beta_\tau-\bar\beta)^2\leq\frac{J(\mathbb{K}_0)-1}d .
\end{equation}
Since $\max_\tau\beta_\tau=1/D>1/\sqrt{5}$, this bounds the largest conjugate $D^\ast$ of $D$, except in the one case where the bound degenerates.
\begin{itemize}
\item \emph{Quadratic fields other than $\mathbb{Q}(\sqrt{5})$.} Here~\eqref{eq:spread} is an equality on the left, and gives $1/D^\ast\geq1/\sqrt{5}-1/\sqrt{\mathfrak{f}}$ with $\mathfrak{f}\geq8$ the discriminant, so $D^\ast\leq10.68$. For $\mathbb{Q}(\sqrt{5})$ the bound degenerates, and Lemma~\ref{lem:dimtwo-sqrt5} replaces it.
\item \emph{Cyclic cubic fields.} Here $J(\mathbb{K}_0)-1=2/\mathfrak{f}\leq\frac{2}{7}$, so $1/D^\ast\geq1/\sqrt{5}-2/\sqrt{21}$ and $D^\ast\leq92.8$.
\item \emph{Quartic fields.} Here $J(\mathbb{K}_0)-1\le\frac{7}{20}$ by Lemma~\ref{lem:dimtwo-fields}, so $1/D^\ast\ge1/\sqrt5-\sqrt{7/40}$ and $D^\ast\le34.7$.
\end{itemize}
All conjugates of $D$ exceed $1$, by \cite[Theorem 4.1.1]{ostrikremarks} applied to the Galois conjugates of $\mathcal{C}$. Let $a_i$ be the elementary symmetric functions of the conjugates of $D$, so that the minimal polynomial is $p(t)=\sum_i(-1)^ia_it^{d-i}$, and let $B=10.7$, $92.8$ or $34.7$ according as $d=2$, $3$ or $4$, so that $D^\ast\leq B$. We enumerate the integer tuples $(a_1,\dots,a_d)$ satisfying the following conditions:
\begin{itemize}
\item Ostrik's $d$-number test $a_d^{\,i}\mid a_i^{\,d}$ for $1\leq i\leq d$ \cite{MR2576705};
\item $m_1=a_{d-1}/a_d\in(1/\sqrt{5},1]$;
\item the bounds $2\le a_d\leq\sqrt5\,B^{d-1}$, $d<a_1\leq\sqrt5+(d-1)B$ and $a_2\leq(d-1)a_1^2/(2d)$, and for $d=4$ also $p(1)>0$, all of which hold when the roots lie in $(1,B]$ and the smallest is below $\sqrt5$.
\end{itemize}
We keep those tuples for which $p$ is irreducible with all roots in $(1,B]$, the smallest below $\sqrt5$, and abelian Galois group, and for these we compute $v(D)$. The computations are carried out by the script \texttt{small\_dimensions\_search.py}, available as an ancillary file; it uses Python with NumPy and SymPy and runs in under a minute. The outcome is as follows.
\begin{itemize}
\item \emph{Quadratic.} For $\mathbb{K}_0=\mathbb{Q}(\sqrt{5})$ the character of conductor $5$ gives $\hat{g}_{\mathbb{K}_0}=1/D-1/D'$, so $v(D)\le1$ is the hypothesis of Lemma~\ref{lem:dimtwo-sqrt5}, which leaves exactly the numbers $D_m$; each has $v(D_m)=1$. For the other real quadratic fields, six polynomials remain, with values $\frac{16}9$, $4$, $4$, $9$, $36$ and $49$.
\item \emph{Cubic.} Of the $7{,}252$ triples enumerated, $868$ have all roots in $(1,92.8]$ with the smallest below $\sqrt5$; of these, $840$ are irreducible and $48$ have cyclic Galois group. Exactly two have value at most $1$, namely $t^3-14t^2+49t-49$ and $t^3-42t^2+245t-343$, both of conductor $7$ and value $1$.
\item \emph{Quartic.} Of the $370{,}181$ tuples enumerated, $103$ have all roots in $(1,34.7]$ with the smallest below $\sqrt{5}$; of these, $79$ are irreducible and $22$ have abelian Galois group. Each of the $22$ has value greater than $1$. The smallest value, $\frac{5}{2}$, is attained by $t^4-30t^3+240t^2-720t+720$, whose splitting field is the cyclic quartic field of conductor $15$: its quadratic character sum vanishes, and $|\hat{g}_{\mathbb{K}_0}(\chi)|^2=\frac{1}{6}$ for the two characters of order $4$.\qedhere
\end{itemize}
\end{proof}

For comparison, Table~\ref{tab:dimtwo-cand} lists the abelian candidates of degree $2$ or $3$ that pass the $d$-number test, have all conjugates in $(1,150]$ and smallest conjugate at most $c_0:=D_4$, and have value at most $2$. The value measures how far each misses Theorem~\ref{thm:orbit}. The table is produced by the same script.

\begin{table}[h]
\centering\footnotesize
\setlength{\tabcolsep}{4pt}
\renewcommand{\arraystretch}{1.3}
\begin{tabular}{lcccc@{\qquad}lcccc}
\toprule
polynomial & $D$ & $\mathfrak{f}$ & $m_1$ & $v(D)$ & polynomial & $D$ & $\mathfrak{f}$ & $m_1$ & $v(D)$\\
\midrule
$t^2-5t+5$ & $1.3820$ & $5$ & $1$ & $\mathbf 1$ & $t^3-56t^2+336t-448$ & $1.9371$ & $7$ & $\frac{3}{4}$ & $\frac{21}{16}$\\
$t^3-14t^2+56t-56$ & $1.5060$ & $7$ & $1$ & $\frac{7}{4}$ & $t^3-182t^2+8281t-15379$ & $1.9389$ & $7$ & $\frac{7}{13}$ & $\frac{301}{169}$\\
$t^2-12t+16$ & $1.5279$ & $5$ & $\frac{3}{4}$ & $\frac{25}{16}$ & $t^3-50t^2+600t-1000$ & $1.9806$ & $7$ & $\frac{3}{5}$ & $\frac{147}{100}$\\
$t^3-30t^2+125t-125$ & $1.5399$ & $7$ & $1$ & $\frac{49}{25}$ & $t^3-63t^2+294t-343$ & $1.9818$ & $9$ & $\frac{6}{7}$ & $\frac{81}{49}$\\
$t^2-30t+45$ & $1.5836$ & $5$ & $\frac{2}{3}$ & $\frac{16}{9}$ & $t^3-150t^2+1800t-3000$ & $1.9934$ & $9$ & $\frac{3}{5}$ & $\frac{189}{100}$\\
$t^2-77t+121$ & $1.6049$ & $5$ & $\frac{7}{11}$ & $\frac{225}{121}$ & $t^2-98t+196$ & $2.0426$ & $5$ & $\frac{1}{2}$ & $\frac{225}{196}$\\
$t^2-84t+144$ & $1.7508$ & $5$ & $\frac{7}{12}$ & $\frac{25}{16}$ & $t^3-102t^2+2601t-4913$ & $2.0505$ & $9$ & $\frac{9}{17}$ & $\frac{567}{289}$\\
$t^2-12t+18$ & $1.7574$ & $8$ & $\frac{2}{3}$ & $\frac{16}{9}$ & $t^2-72t+144$ & $2.0589$ & $8$ & $\frac{1}{2}$ & $\frac{16}{9}$\\
$t^3-98t^2+392t-392$ & $1.7634$ & $7$ & $1$ & $\frac{7}{4}$ & $t^3-196t^2+9604t-19208$ & $2.0880$ & $7$ & $\frac{1}{2}$ & $\frac{43}{28}$\\
$t^3-45t^2+486t-729$ & $1.7826$ & $7$ & $\frac{2}{3}$ & $\frac{49}{27}$ & $t^2-40t+80$ & $2.1115$ & $5$ & $\frac{1}{2}$ & $\mathbf 1$\\
$t^3-14t^2+49t-49$ & $1.8412$ & $7$ & $1$ & $\mathbf 1$ & $t^3-42t^2+245t-343$ & $2.1558$ & $7$ & $\frac{5}{7}$ & $\mathbf 1$\\
$t^3-36t^2+180t-216$ & $1.8479$ & $7$ & $\frac{5}{6}$ & $\frac{49}{36}$ & $t^3-108t^2+2916t-5832$ & $2.1711$ & $9$ & $\frac{1}{2}$ & $\frac{7}{4}$\\
$t^3-24t^2+144t-192$ & $1.8716$ & $9$ & $\frac{3}{4}$ & $\frac{27}{16}$ & $t^3-55t^2+726t-1331$ & $2.1787$ & $7$ & $\frac{6}{11}$ & $\frac{147}{121}$\\
$t^2-91t+169$ & $1.8967$ & $5$ & $\frac{7}{13}$ & $\frac{225}{169}$ & $t^2-105t+225$ & $2.1885$ & $5$ & $\frac{7}{15}$ & $\mathbf 1$\\
$t^2-15t+25$ & $1.9098$ & $5$ & $\frac{3}{5}$ & $\mathbf 1$ &  & & & & \\
\bottomrule
\end{tabular}
\caption{Abelian candidates of degree $2$ or $3$ with all conjugates in $(1,150]$, smallest conjugate $D\leq c_0=D_4$, and value $v(D)=\mathrm{max}_{\chi\neq1}\mathfrak{f}(\chi)|\hat{g}_{\mathbb{K}_0}(\chi)|^2$ at most $2$, ordered by $D$. The six with value $1$ are $\alpha_3$, $\alpha_5$ and $D_1,\dots,D_4$. By Lemma~\ref{lem:dimtwo-values}, the only candidates with value at most $1$ and $c_0<D<\sqrt{5}$ are the $D_m$ with $m\geq5$.}
\label{tab:dimtwo-cand}
\end{table}

\begin{proposition}\label{prop:realize}
The numbers $1$, $D_1$, $\alpha_3$, $D_2$, $2$, $D_3$ and $\alpha_5$ are global dimensions of spherical fusion categories.
\end{proposition}

\begin{proof}
By Section~\ref{sec:prelim}, Galois conjugates of spherical fusion categories are spherical, so it suffices to realize one Galois conjugate of each number. The categories $\mathrm{Vec}$, $\mathcal{YL}$, $\mathcal{YL}\boxtimes\mathcal{YL}$ and $\mathrm{Vec}_{\mathbb{Z}/2}$ have global dimensions $1$, $D_1$, $D_1^2=D_2$ and $2$. The category $\mathcal{C}(\mathfrak{sl}_2,5)_{\mathrm{ad}}$ has global dimension $7/(4\sin^2(\pi/7))\approx9.29590$, a root of $t^3-14t^2+49t-49$ \cite[Section 4.3]{MR3427429}. The even part of the 2D2 subfactor \cite{MR4939540,Izu18,MP15} has simple objects $\mathbbm{1},\alpha,\eta,\alpha\eta$ with $\eta\otimes\eta\cong\mathbbm{1}\oplus2\eta\oplus2\alpha\eta$, so $\mathrm{dim}(\eta)=2+\sqrt5$ and its global dimension is $2+2(2+\sqrt{5})^2=20+8\sqrt{5}$, the conjugate of $D_3$. Finally, the modular category $\C(\mathfrak{sl}_3,4)$ has global dimension $1/S_{\mathbbm{1},\mathbbm{1}}^2$ with $S_{\mathbbm{1},\mathbbm{1}}=\frac{8}{7\sqrt{3}}\sin^2(\pi/7)\sin(2\pi/7)$ \cite{BK,RowellQG}, and its adjoint subcategory, the trivial component of its $\mathbb{Z}/3$-grading, has global dimension one third of this, approximately $35.34242$. This is a root of $t^3-42t^2+245t-343$, whose roots are $2.15585$, $4.50173$ and $35.34242$.
\end{proof}

\begin{table}[h]
\centering\small
\begin{tabular}{llc>{\raggedright\arraybackslash}p{4.6cm}}
\toprule
$\mathrm{dim}(\mathcal{C})$ & minimal polynomial & $m_1$ & known spherical fusion categories of this dimension\\
\midrule
$1$ & $t-1$ & $1$ & $\mathrm{Vec}$\\
$D_1=\frac{5-\sqrt{5}}{2}\approx1.38197$ & $t^2-5t+5$ & $1$ & $\mathcal{YL}$\\
$\alpha_3\approx1.84117$ & $t^3-14t^2+49t-49$ & $1$ & Galois conjugates of $\mathcal{C}(\mathfrak{sl}_2,5)_{\mathrm{ad}}$\\
$D_2=\frac{15-5\sqrt{5}}{2}\approx1.90983$ & $t^2-15t+25$ & $\frac{3}{5}$ & $\mathcal{YL}\boxtimes\mathcal{YL}$\\
$2$ & $t-2$ & $\frac{1}{2}$ & the pointed categories with group $\mathbb{Z}/2$\\
$D_3=20-8\sqrt{5}\approx2.11146$ & $t^2-40t+80$ & $\frac{1}{2}$ & Galois conjugates of the even part of the 2D2 subfactor\\
$\alpha_5\approx2.15585$ & $t^3-42t^2+245t-343$ & $\frac57$ & a Galois conjugate of $\mathcal{C}(\mathfrak{sl}_3,4)_{\mathrm{ad}}$\\
$D_4\approx2.18847$ & $t^2-105t+225$ & $\frac{7}{15}$ & unknown\\
$D_m\ (m\geq5)$ & $t^2-5F_{2m}t+5(L_{2m}-2)$ & $\frac{F_{2m}}{L_{2m}-2}$ & unknown\\
\bottomrule
\end{tabular}
\caption{The possible global dimensions of spherical fusion categories below $\sqrt{5}$. Here $m_1$ is the mass of the dimension orbit. The $D_m$ increase to $\sqrt{5}$.}
\label{tab:dimtwo}
\end{table}

\begin{proof}[Proof of Theorem~\ref{thm:dimtwo}]
The rational case was treated above. Let $D<\sqrt{5}$ be irrational and let $D_0\leq D$ be its smallest conjugate. Then $D_0$ is the global dimension of a Galois conjugate of $\mathcal{C}$, which is again spherical (Section~\ref{sec:prelim}), so Lemma~\ref{lem:dimtwo-values} gives $D_0\in\{\alpha_3,\alpha_5\}\cup\{D_m:m\geq1\}$. The other conjugates of these numbers all exceed $\sqrt5$: those of $\alpha_3$ are $2.86294$ and $9.29590$, those of $\alpha_5$ are $4.50173$ and $35.34242$, and that of $D_m$ is $\sqrt{5}(\phi^{2m}-1)\ge\sqrt{5}\,\phi^2-\sqrt{5}=\sqrt{5}\,\phi>\sqrt{5}$. Hence $D=D_0$. The realizations are Proposition~\ref{prop:realize}.
\end{proof}

\begin{corollary}\label{cor:ostrik}
For every $c<\sqrt{5}$ the set $X_s\cap[1,c]$ is finite. In particular $\sqrt{2}$ is not a limit point of $X_s$, the interval $(\frac{1}{2}(5-\sqrt{5}),\sqrt{2})$ contains no point of $X_s$, and the smallest point of $X_s$ larger than $\frac{1}{2}(5-\sqrt{5})$ is $\alpha_3$.
\end{corollary}

\begin{proof}
By Theorem~\ref{thm:dimtwo}, $X_s\cap[1,c]$ is contained in $\{1,2,\alpha_3,\alpha_5\}\cup\{D_m:D_m\leq c\}$, which is finite because $D_m$ increases to $\sqrt{5}$. The remaining statements follow from $1<D_1<\sqrt2<\alpha_3<D_2<\dots$ and $\alpha_3\in X_s$.
\end{proof}

\begin{note}
The large conjugate $\sqrt{5}\,(\phi^{2m}-1)$ of $D_m$ is $5F_m\phi^m$ for $m$ even and $\sqrt{5}\,L_m\phi^m$ for $m$ odd. For $m=1,2,3$ these are the global dimensions $\phi\sqrt{5}\approx3.618$ of the Fibonacci category, $5\phi^2\approx13.09$ of its square, and $4\sqrt{5}\,\phi^3\approx37.89$ of the even part of 2D2. \end{note}

\begin{corollary}\label{arbcor}
    If $\mathcal{C}\not\simeq\mathrm{Vec}$ is a fusion category, then $\mathrm{dim}(\mathcal{C})\geq\frac{1}{2}D_3=10-4\sqrt{5}$.
\end{corollary}

\begin{proof}
Recall the sphericalization $\tilde{\mathcal{C}}$ of a fusion category $\mathcal{C}$ \cite[Remark 3.1]{MR2183279}. Here $\mathrm{dim}(\mathcal{C})=\sum_X|X|^2$ is the global dimension of \cite{MR2183279}. If $\mathcal{C}\not\simeq\mathrm{Vec}$ then $\mathrm{dim}(\mathcal{C})>1$, since $|X|^2>0$ for every simple $X$ \cite[Theorem 2.3]{MR2183279}. The sphericalization $\tilde{\mathcal{C}}$ is a spherical fusion category with $\mathrm{dim}(\tilde{\mathcal{C}})=2\mathrm{dim}(\mathcal{C})>2$ \cite[Remark 3.1]{MR2183279}. By Theorem~\ref{thm:dimtwo}, no spherical fusion category has global dimension in $(2,D_3)$, so $2\mathrm{dim}(\mathcal{C})\geq D_3=20-8\sqrt{5}$.
\end{proof}

\bibliographystyle{plainurl}
\bibliography{bib}
\end{document}